\documentclass[reqno]{amsart}

\usepackage{graphicx} 
\usepackage{enumitem}
\usepackage{amsmath,amsfonts,amssymb,amsthm,amscd,latexsym,cite}
\usepackage{mathrsfs}
\usepackage{color}
\newtheorem{thm}{Theorem}[section]
\newtheorem{theorem}[thm]{Theorem}

\newtheorem{lemma}[thm]{Lemma}

\newtheorem{proposition}[thm]{Proposition}

\newtheorem{problem}{Problem}

\newcommand{\cB}{{\mathcal B}}

\newcommand{\cH}{{\mathcal H}}

\newcommand{\cM}{{\mathcal M}}
\newcommand{\cN}{{\mathcal N}}

\newcommand{\cU}{{\mathcal U}}

 \usepackage{color}
  
 \newcommand{\norm}[1]{\left\lVert#1\right\rVert}
\usepackage{hyperref}					
\hypersetup{colorlinks,
	linkcolor=blue,%
	citecolor=blue}
\title{Tingley's Problem for Schatten \(p\)-Classes, $0<p\ne 2<\infty$}
\author[Jinghao Huang]{J. Huang}\thanks{
The authors were supported by the NNSF of China (12031004, 12301160 and 12471134).}
\address{Institute for Advanced Study in Mathematics of HIT, Harbin Institute of Technology, Harbin, 150001, China}
\email{{\color{blue}jinghao.huang@hit.edu.cn}}

\author[Yunpeng Zhu]{Y. Zhu} 
\address{Institute for Advanced Study in Mathematics of HIT, Harbin Institute of Technology, Harbin, 150001, China}
\email{{\color{blue}24s012030@stu.hit.edu.cn}}

\begin{document}
\begin{abstract} 
We give an affirmative answer to  Tingley's problem for Schatten--von Neumann classes $C_p$ by a unified approach  for all $0<p\ne 2<\infty$, extending   results due to Fern\'andez-Polo et al.
\end{abstract}



\maketitle

\section{Introduction}
\subsection{Background}
In 1987,  
Tingley posed the following problem\cite{Tingley}, which has since become known as Tingley's problem \cite{Ding1,Ding2009}.
\begin{problem}
    Let \(X\) and \(Y\) be real normed spaces, and let \(V_0\) be a surjective isometry between the unit spheres of \(X\) and \(Y\) (i.e., $\norm{V_0(x) -V_0(y)}_Y =\norm{x-y}_X$). Does there exist a surjective real linear isometry \(T: X \to Y\) extending \(V_0\), that is, \(T(x) = V_0(x)\) for all \(x\) in the unit sphere of \(X\)?
\end{problem}
Since then, Tingley's problem has attracted considerable attention and remains
open in full generality. Nevertheless, positive answers have been obtained for
many important classes of spaces. In the commutative setting, the problem has
been studied for classical Banach spaces such as \(\ell_p\) and \(L_p\)-spaces, $1\le p\ne 2\le \infty$ (see
\cite{Ding1,Ding2,Ding3,Tan1,Tan2,Tan3}), while the quasi-Banach case when $0<p<1$ was treated in \cite{Tan2,Li}. 
 In fact, \(L_p\)-spaces satisfy an even
stronger extension property, usually called the Mazur--Ulam property\cite{CD11}: every
surjective isometry from the unit sphere of an \(L_p\)-space onto the unit
sphere of an arbitrary real Banach space extends to a surjective real linear
isometry between the whole spaces.

\subsection{Tingley's property for noncommutative $L_p$-spaces}
Tanaka~\cite{Tanaka},    Fern\'andez-Polo,   Peralta \cite{Polo3,Polo18}, and   Ozawa and   Mori \cite{Mori,MO20} proved that Tingley's problem has an affirmative answer for all unital $C^*$-algebras (see \cite{LNW,WN22} and \cite{BCFP,Polo18b,KP} for recent progress concerning Tingley's problem for $C^*$-algebras and JBW$^*$-triples, respectively).

In 2018, Mori asked the following question. 
\begin{problem}\cite[Problem 6.3]{Mori}
    Let $1 < p\ne 2 < \infty$, and let $\cM, \cN$ be von
Neumann algebras and $T : S(L_p(\cM)) \to S(L_p(\cN))$ be a surjective isometry between the unit spheres of
(Haagerup) noncommutative $L_p$-spaces (with respect to fixed normal semifinite faithful weights). Does \(T\)
admit an extension to a real linear surjective isometry $T : L_p(\cM) \to L_p(\cN)$?
\end{problem} 

The special case of 
the trace class   \(C_1(H)\) was treated in \cite{Polo2}, and the 
case of general noncommutative $L_1$-spaces was proved in \cite{Mori}. 
Tingley's problem for the Schatten \(p\)-classes \(C_p(H)\), 
$1<p\ne 2<\infty$, was settled in \cite{Polo} by an approach that is completely different from those used in \cite{Polo2,Mori}. 



In this paper, we provide a unified solution of Tingley's problem for Schatten
\(p\)-classes for all \(0<p\neq2<\infty\) by using Weyl's submajorization (i.e., logarithmic submajorization).

\begin{theorem}\label{main}
    Let \(H_1,H_2\) be complex Hilbert spaces, and let \(V_0: S(C_p(H_1)) \to S(C_p(H_2))\) be a surjective isometry, where \(0<p\neq 2<\infty\). Then there exists a surjective complex linear or conjugate linear isometry that extends \(V_0\). 
\end{theorem}

It should be noted that 
our approach covers the quasi-Banach case when \(0<p<1\), while
the convexity used essentially in 
the arguments in \cite{Polo} is no
longer applicable. Instead, our
approach is based on a direct spectral analysis of finite-dimensional
matrices, together with (logarithmic) submajorization 
techniques. Below, we outline the proof of Theorem \ref{main}.

As in the proof of \cite{Polo}, we will need to prove that 
the isometry \(V_0\) preserves orthogonality (i.e., 
 \(
 xy^*=0\) and \(x^*y=0\)) and 
one of the following statements holds for every minimal partial isometry \(e\) (denoted by  \(e\in \cU_{\min}(H)\); for simplicity, we assume that $H=H_1=H_2$),  and every \(\lambda\in \mathbb{C}\) satisfies \(|\lambda|=1\) :
\[
   V_0(\lambda e)=\lambda V_0(e) \text{ or } V_0(\lambda e)=\bar{\lambda} V_0(e),
\]
(see Lemma \ref{orth-preserving} and Proposition \ref{U_min preserving} below).


A key ingredient in proving Theorem \ref{main} is a metric characterization of \({\rm Tr}(e^*v)\) for \(e,f\in \cU_{\min}(H)\), i.e., for
\(e,v,e',v'\in \mathcal U_{\min}(H)\),  
\begin{align}\label{e-v to tr(ev)}
   \|\lambda e-v\|_p=\|\lambda e'-v'\|_p,
   \quad \lambda\in\mathbb T \Rightarrow  {\rm Tr}(e^*v)={\rm Tr}((e')^*v'),
\end{align}
 (see Lemma \ref{lem:trace-recovery} below).
This is the main point at which our argument differs from \cite{Polo}: there, the trace identity is deduced after solving the
two-dimensional Tingley's problem for \(C_p\), where \(1<p\neq 2<\infty\). Here \eqref{e-v to tr(ev)} follows directly from a two-dimensional spectral computation.




Another important component of the proof of Theorem \ref{main}
is a characterization of
rank-one perturbations,
which extends \cite[Proposition 2.10]{Polo} to  the setting of 
\(0<p\ne 2 <\infty\). More precisely, we prove that, for a finite-dimensional complex Hilbert space \(H\), if
\(x,y\in S(C_p(H))\), where \(0<p\neq 2<\infty\), then
\begin{align}\label{rank-one perturbations}
    \|x-e\|_p=\|y-e\|_p
    \text{ for all } e\in \mathcal U_{\min}(H)\Rightarrow  x=y,
\end{align}
(see Proposition \ref{rank-one partial isometry decide} below). 
We first establish 
the above result for \(x,y\in S(C_p(H)_+)\),
\[
\|x-e\|_p=\|y-e\|_p
   \text{ for all rank-one projection~} e \Rightarrow  x=y,
\]
(see Proposition \ref{rank-one decide+} below), extending \cite[Lemma]{Nagy} via
the Rayleigh quotient theorem \cite[Theorem 4.2.2]{Horn} and Cauchy's
interlacing theorem \cite[Theorem 1]{Cauchy}.
We then combine a logarithmic submajorization inequality for positive
semidefinite matrices \cite[Theorem 2]{zhan} with submajorization arguments to derive a characterization of positive invertible matrices with respect to $\norm{\cdot}_p$
(see
Lemma \ref{positive decide}). 
Now, a standard argument involving 
 polar decomposition and unitary invariance of \(\norm{\cdot}_p\) yields \eqref{rank-one perturbations} for arbitrary elements of \(S(C_p(H))\). 

By \eqref{e-v to tr(ev)} and \eqref{rank-one perturbations},
arguing as in \cite{Polo} and combining with \cite[Theorem 2]{Molnar},
we can extend \(V_0\) to a surjective isometry on the whole space $C_p(H)$. 


    
In Section \ref{s:p}, we also consider Tingley's problem for the positive sphere of \(C_p(H)\), \(0<p<\infty\) (see Proposition \ref{positive sphere} below).

\section{Preliminaries}

In this section,
we recall the notions needed
in this paper. 
General information concerning   Schatten $p$-classes can be found in
\cite{Horn,Simon,Gohberg}.

Let \(H\) be a complex (not necessarily separable) Hilbert space with identity \(I\). For a compact operator \(x\in \mathcal{B}(H)\) (the algebra of all bounded linear operators on $H$),  
the singular values of \(x\) are defined as the eigenvalues of \(|x|\) (see \cite[p.~3]{Cauchy}), listed
in non-increasing order and repeated according to multiplicity:
\[
   s_1(x)\ge s_2(x)\ge \cdots \ge 0.
\]
For a sequence $x$, we denote by 
\(s(x)\)  the  non-increasing   rearrangement of $|x|$.

For \(0<p<\infty\), the Schatten \(p\)-class (see \cite[p.~3]{Cauchy}) is defined by 
\[
   C_p(H)=\left\{x\in \mathcal{B}(H): x \text{ is compact and }
   \sum_{j=1}^{\infty} s_j(x)^p<\infty \right\},
\]
and \(\left\|\cdot\right\|_p\) on \(C_p(H)\) is defined by 
\[
   \|x\|_p
   =
   \left(\sum_{j=1}^{\infty} s_j(x)^p\right)^{1/p}
   =
   \bigl({\rm Tr}(|x|^p)\bigr)^{1/p}, \quad x\in C_p(H),
\]
where ${\rm Tr}$ stands for the standard trace on $B(H)$.
When \(1\le p<\infty\), \(\left \|\cdot \right\|_p\) is a norm and \(C_p(H)\) is a Banach
space. 
In particular, \(C_1(H)\) is the trace class and \(C_2(H)\) is the
Hilbert--Schmidt class (see \cite{Simon}).
When \(0<p<1\), the same formula defines a quasi-norm instead of a
norm. 

An element \(e\in \mathcal{B}(H)\) is called a projection if \(e^2=e=e^*\) and
an element \(e\in \mathcal{B}(H)\) is called a partial isometry (see \cite[p.~18]{DPS}) if 
\(e^*e,ee^*\)
are projections.

For unit vectors \(\eta,\xi\in H\), the operator
\(\eta\otimes\xi\), defined by
\[
   (\eta\otimes\xi)(a)=\langle a,\xi\rangle\eta,
   \qquad a\in H,
\]
is a rank-one partial isometry (a minimal partial isometry). Conversely, every minimal partial isometry in
\(\mathcal{B}(H)\) can be written in this form. Moreover,
every element \(x\in C_p(H)\) admits a singular value decomposition of the form
\[
   x=\sum_{n=1}^{\infty}s_n(x)e_n,
\]
where \(\{s_n\}\subset \mathbb R^+\setminus \{0\}\) is the sequence of nonzero singular values of \(x\) and \(e_n\) are mutually orthonormal minimal partial isometries \cite[Theorem 1.4]{Simon}. 

Throughout this paper, \(P(H)\) denotes the set of all projections on \(H\), and
\(\operatorname{Proj}_1(H)\) denotes the set of all minimal, equivalently
rank-one projections on \(H\). 
Moreover, \(\mathcal U_{\min}(H)\) denotes the set of all
minimal partial isometries on \(H\), which are contained in \(C_p(H)\) for \(0<p<\infty\), and
\(
   \mathbb T=\{\lambda\in\mathbb C: |\lambda|=1\}
\)
denotes the unit circle of \(\mathbb C\).


For $x, y\in C_p(H)$, we say that $x$ is orthogonal to $y$, denoted by $x\perp y$, if  \[xy^*=0,\quad x^*y=0.\] Moreover, let \( \ell(x)\) and \( r(x)\) denote its left and
right support projections, respectively. 

\section{Proof of the main result}


The following lemma was proved in 
  \cite[Lemma 2.2]{Polo2} for \(p=1\) and in \cite[Lemma 2.3]{Polo} for 
    \(1<p\neq 2<\infty\). 

\begin{lemma}\label{orth-preserving}
    Let \(H_1,H_2\) be complex Hilbert spaces, and let \(V_0: S(C_p(H_1)) \to S(C_p(H_2))\) be a surjective isometry, where $0<p\ne 2<\infty$. For any $x,y\in S(C_p(H_1))$,  \(x\perp y\) if and only if \(V_0(x)\perp V_0(y)\). Moreover, \(\dim(H_1) = \dim(H_2)\).
\end{lemma}
\begin{proof}
  It suffices to prove the case when \(0< p< 1\).  
  Let \(x,y\in S(C_p(H_1))\) satisfy \(x\perp y\). By Lemma \ref{orth for p<1}, we have \(\|x-y\|_p^p=\|x\|_p^p+\|y\|_p^p\). Since \(V_0\) is an isometry, we have \(\|V_0(x)-V_0(y)\|_p^p=\|V_0(x)\|_p^p+\|V_0(y)\|_p^p\). By Lemma \ref{orth for p<1} again, \(V_0(x)\perp V_0(y)\). The converse follows by applying the same argument to \(V_0^{-1}\).

    Since $V_0$ preserves orthogonality of minimal projections, and since the dimensions of \(H_1,H_2\) are precisely the cardinalities of \(Proj_1(H_1),Proj_1(H_2)\), respectively (see e.g. \cite[p.~4]{Polo}), 
    it follows that \( \dim(H_1) \le \dim(H_2)\).
    Similarly, we have \( \dim(H_2) \le \dim(H_1)\).
    Hence, \(\dim(H_1) = \dim(H_2)\). 
\end{proof}

The following proposition was proved in \cite[Propositions 2.4 and 2.7]{Polo} for \(1<p\neq 2<\infty\).
The case \(p=1\) follows from \cite[Lemma 2.2]{Polo2}.
\begin{proposition}\label{U_min preserving}
    Let \(H_1,H_2\) be two complex Hilbert spaces, and let \(V_0: S(C_p(H_1)) \to S(C_p(H_2))\) be a surjective isometry, where \(0<p\neq 2<\infty\). Then the following
    statements hold:
    \begin{enumerate}[label=(\roman*)]
        \item \(V_0(\mathcal U_{\min}(H_1))=\mathcal U_{\min}(H_2).\)
        \item For each \(e\in \mathcal U_{\min}(H_1)\), we have \(V_0(\lambda e)=\lambda V_0(e)\) or \(V_0(\lambda e)=\overline{\lambda} V_0(e)\) for all \(\lambda \in \mathbb{T}\).
        \item For a fixed \(e_0\in \cU_{\min}(H_1)\), if \(V_0(\lambda_0 e_0)=\lambda_0 V_0(e_0)\)  (respectively \(V_0(\lambda_0 e_0)=\bar{\lambda}_0 V_0(e_0)\)) for some \(\lambda_0\in\mathbb T\setminus \{-1,1\}\),  then \(V_0(\lambda e_0)=\lambda V_0(e_0)\)  (respectively \(V_0(\lambda e_0)=\bar{\lambda} V_0(e_0)\)) for all \(\lambda\in\mathbb T\).
        \item If there exists \(e_0\in \mathcal U_{\min}(H_1)\) and \(\lambda_0 \in \mathbb T\setminus \{-1,1\}\) such that \(V_0(\lambda_0 e_0)=\lambda_0 V_0(e_0)\) (respectively \(V_0(\lambda_0 e_0)=\overline{\lambda_0} V_0(e_0)\)), then for arbitrary \(e\in \mathcal U_{\min}(H_1)\) and \(\lambda \in \mathbb{T}\), we have that \(V_0(\lambda e)=\lambda V_0(e)\) (respectively \(V_0(\lambda e)=\overline{\lambda} V_0(e)\)).
    \end{enumerate}
\end{proposition}
\begin{proof}
It suffices to prove the case when \(0< p\le 1\).
The proofs of  (i)-(iii) are  the same as the proof of \cite[Proposition 2.4.(a)-(c)]{Polo}.

 Setting  \(\lambda=-1\) in Proposition \ref{U_min preserving} (ii), we have that 
\begin{align}\label{-e}
V_0(-e)=-V_0(e)
\end{align}
for all \(e\in \mathcal U_{\min}(H_1)\). Then, we have that
\begin{align}\label{e+f}
\left\|V_0(e)+V_0(f)\right\|_p^p
&=\left\|V_0(e)-(-V_0(f))\right\|_p^p
\stackrel{\eqref{-e}}{=}\left\|V_0(e)-(V_0(-f))\right\|_p^p
\\&=\left\|e-(-f)\right\|_p^p=\left\|e+f\right\|_p^p,\nonumber
\end{align}
for all \(e,f\in \cU_{\min}(H_1)\).

Define
\[
\mathfrak{D}_1
:=
\left\{
v\in \mathcal{U}_{\min}(H_1):
V_0(\lambda v)=\lambda V_0(v)
\text{ for all } \lambda\in \mathbb{T}
\right\}
\]
and
\[
\mathfrak{D}_2
:=
\left\{
v\in \mathcal{U}_{\min}(H_1):
V_0(\lambda v)=\overline{\lambda}V_0(v)
\text{ for all } \lambda\in \mathbb{T}
\right\}.
\]
By Proposition \ref{U_min preserving} (ii), we have that
\[
\mathcal{U}_{\min}(H_1)=\mathfrak{D}_1\dot{\cup}\mathfrak{D}_2 ,
\]
where \(\dot{\cup}\) stands for the disjoint union.
Since \(\cU_{\min}(H_1)\) is connected  \cite[Theorem 7]{Halmos}, it suffices to prove that \(\mathfrak{D}_1\) is open. For a fixed 
\(v\in \mathfrak{D}_1\), let \(w\in \mathcal{U}_{\min}(H_1)\) satisfy
\begin{align}\label{varepsilon}
\|v-w\|_p^p<2^{p-1}\varepsilon
\end{align}
for some \(0<\varepsilon<1\).
By Proposition \ref{U_min preserving} (ii), we have that
\(V_0(iw)=iV_0(w)\) or \(V_0(iw)=-iV_0(w)\).
Assume that 
\begin{align}\label{-iw}
   V_0(iw)=-iV_0(w).
\end{align}
Since \(v\in \mathfrak{D}_1\), we have that
\begin{align}\label{iv}
   V_0(iv)=iV_0(v).
\end{align}
Hence, we have that
\begin{align*}
\|v-w\|_p^p
&=
\|iv-iw\|_p^p 
=
\|V_0(iv)-V_0(iw)\|_p^p  
\stackrel{\eqref{-iw},\eqref{iv}}{=}
\|iV_0(v)+iV_0(w)\|_p^p  
\\&=
\|V_0(v)+V_0(w)\|_p^p  
\stackrel{\eqref{e+f}}{=}
\|v+w\|_p^p,
\end{align*}
which implies that
\begin{align}\label{v+w}
\|v+w\|_p^p=\|v-w\|_p^p\stackrel{\eqref{varepsilon}}{<}2^{p-1}\varepsilon
\end{align}
Using  \cite[Theorem 2.8]{McCarthy}, we have
\begin{align*}
1
=
\|v\|_p^p
=
\left\|
\frac{v-w}{2}
+
\frac{v+w}{2}
\right\|_p^p 
\le
\left\|
\frac{v-w}{2}
\right\|_p^p
+
\left\|
\frac{v+w}{2}
\right\|_p^p 
\stackrel{\eqref{v+w}}{<}
\varepsilon
<
1,
\end{align*}
which is impossible. Hence
\[
V_0(iw)=iV_0(w).
\]
By Proposition \ref{U_min preserving} (iii), it follows that
\[
w\in \mathfrak{D}_1 .
\]
Therefore, \(\mathfrak{D}_1\) is open. 
\end{proof}

\begin{lemma}\label{lem:Phi-monotone}
Let \(a\in\mathbb{R}^+\) and \(p>0\). Define \(\Phi_a:[2\sqrt{a},\infty)\to \mathbb{R}\) by
\[
\Phi_a(T):=
\left(\frac{T+\sqrt{T^2-4a}}{2}\right)^{p/2}
+
\left(\frac{T-\sqrt{T^2-4a}}{2}\right)^{p/2}.
\]
Then \(\Phi_a\) is strictly increasing on \([2\sqrt{a},\infty)\).
\end{lemma}

\begin{proof}
Set
\[
r:=\frac{p}{2}>0.
\]

If \(a=0\), then for \(T\ge 0\), we have 
\[
\Phi_0(T)
=
\left(\frac{T+\sqrt{T^2}}{2}\right)^r
+
\left(\frac{T-\sqrt{T^2}}{2}\right)^r=
T^r.
\]
Since \(r>0\), the function \(T\mapsto T^r\) is strictly increasing on \([0,\infty)\). Therefore, \(\Phi_0\) is strictly increasing on \([0,\infty)\).

Now, assume that \(a>0\). 
Note that \(\cosh u\) is strictly increasing on \([0,\infty)\) and its range is \([1,\infty)\).
For each \(T\ge 2\sqrt{a}\), there exists a unique \(u\ge 0\) such that
\begin{align*}
   T=2\sqrt{a}\cosh u.
\end{align*}
  Then we obtain
\begin{align*}
   \sqrt{T^2-4a}
=
\sqrt{4a\cosh^2u-4a}
=
2\sqrt{a}\sinh u. 
\end{align*}
Hence,
\begin{align*}
\Phi_a(T)
=&
\left(\frac{2\sqrt{a}\cosh u+2\sqrt{a}\sinh u}{2}\right)^r+\left(\frac{2\sqrt{a}\cosh u-2\sqrt{a}\sinh u}{2}\right)^r
\\=&
(\sqrt{a}\,e^u)^r+(\sqrt{a}\,e^{-u})^r
=
a^{r/2}(e^{ru}+e^{-ru})
=
2a^{r/2}\cosh(ru).
\end{align*}

Since \(r>0\), the function \(u\mapsto \cosh(ru)\) is strictly increasing on \([0,\infty)\). Moreover, \(T=2\sqrt{a}\cosh u\) is also strictly increasing with respect to \(u\). It follows that \(\Phi_a(T)\) is strictly increasing for \(T\) in \([2\sqrt{a},\infty)\).
\end{proof}

The following lemma should be compared with \cite[Corollary 2.13]{Polo}. 
\begin{lemma}\label{lem:trace-recovery}
Let \(H\) be a complex Hilbert space, and let \(0<p<\infty\). Let
\(e,v,e',v'\in \mathcal U_{\min}(H)\). If
\begin{align}\label{e,e'}
   \|\lambda e-v\|_p=\|\lambda e'-v'\|_p,
   \qquad \lambda\in\mathbb T ,
\end{align}
then
\[
   {\rm Tr}(e^*v)={\rm Tr}((e')^*v').
\]
\end{lemma}

\begin{proof}
Put
\begin{align}\label{z,z'}
   z={\rm Tr}(e^*v),
   \qquad
   z'={\rm Tr}((e')^*v').
\end{align}
{\bf{Step 1}}
We first compute the function
\[
   \lambda\mapsto \|\lambda e-v\|_p^p .
\]

Write
\[
   e=\eta\otimes \xi,\qquad v=\zeta\otimes \omega,
\]
where \(\eta,\xi,\zeta,\omega\) are unit vectors in \(H\). Choose orthonormal bases
\(\{\eta,\eta^\perp\}\) and \(\{\xi,\xi^\perp\}\) in the two-dimensional spaces
generated by \(\eta,\zeta\) and by \(\xi,\omega\), respectively. Write
\[
   \zeta=a\eta+b\eta^\perp,
   \qquad
   \omega=c\xi+d\xi^\perp ,
\]
with
\[
   |a|^2+|b|^2=1,
   \qquad
   |c|^2+|d|^2=1.
\]
With respect to these bases, \(e\) and \(v\) have the matrix forms
\[
   e=
   \begin{pmatrix}
      1&0\\
      0&0
   \end{pmatrix},
   \qquad
   v=
   \begin{pmatrix}
      a\bar c&a\bar d\\
      b\bar c&b\bar d
   \end{pmatrix}.
\]
Hence, 
\begin{align}\label{lambda e-v}
   \lambda e-v
   =
   \begin{pmatrix}
      \lambda-a\bar c&-a\bar d\\
      -b\bar c&-b\bar d
   \end{pmatrix}.
\end{align}
Let \(t_1(\lambda),t_2(\lambda)\) be the two eigenvalues of
\((\lambda e-v)^*(\lambda e-v)\). Let
\begin{align}\label{T(lambda)}
   T(\lambda)=t_1(\lambda)+t_2(\lambda)
   =
   {\rm Tr}\bigl((\lambda e-v)^*(\lambda e-v)\bigr).
\end{align}
Since
\[
   (\lambda e-v)^*(\lambda e-v)
   =
   e^*e+v^*v-\bar\lambda e^*v-\lambda v^*e,
\]
we have that
\begin{align}\label{T lambda}
   T(\lambda)
   \stackrel{\eqref{z,z'}}{=}
   2-2\operatorname{Re}(\bar\lambda z).
\end{align}
Let
\begin{align}\label{D(lambda)}
D(\lambda)=t_1(\lambda)t_2(\lambda),
\end{align}
Combining this with
\[
   \det(\lambda e-v)\stackrel{\eqref{lambda e-v}}{=}-\lambda b\bar d,
\]
we have that
\[
D(\lambda)=\det\bigl((\lambda e-v)^*(\lambda e-v)\bigr)
   =
   |\det(\lambda e-v)|^2=|b|^2|d|^2,
\]
which means that \(D(\lambda)\) is a constant which depends on \(e,v\) only. Let
\[
   D=|b|^2|d|^2.
\]
Hence,
\begin{eqnarray*}
    \|\lambda e-v\|_p^p
    &=&
    t_1(\lambda)^{p/2}+t_2(\lambda)^{p/2}
   \\& \stackrel{\eqref{T(lambda)},\eqref{D(lambda)}}{=}&
\left(\frac{T(\lambda)+\sqrt{T(\lambda)^2-4D}}{2}\right)^{p/2}
+
\left(\frac{T(\lambda)-\sqrt{T(\lambda)^2-4D}}{2}\right)^{p/2}
\nonumber
\\&= & \Phi_{D}(T(\lambda))\nonumber
\end{eqnarray*}
where \(\Phi\) is the function in Lemma \ref{lem:Phi-monotone}. Similarly,
\[
   \|\lambda e'-v'\|_p^p
   =
   \Phi_{D'}(T'(\lambda)),
\]
where
\[
   D'=
   |b'|^2|d'|^2,
   \quad
   T'(\lambda)=2-2\operatorname{Re}(\bar\lambda z').
\]

{\bf{Step 2}} Now,  we prove that \(z=z'\).

By \eqref{e,e'}, we have
\begin{align}\label{Phi,trace}
\Phi_{D}(T(\lambda))=\Phi_{D'}(T'(\lambda)),
\end{align}
for all \(\lambda\in\mathbb{T}\).

If \(z=0\), then
\[\Phi_{D}(T(\lambda))\stackrel{\eqref{T lambda}}{=}
   \Phi_D\bigl(2-2\operatorname{Re}(\bar\lambda z)\bigr)
   =
   \Phi_D(2)
\]
is a constant. By \eqref{Phi,trace}, we have that \(\Phi_{D'}(T'(\lambda))\) is a constant as well. 
Since \(\Phi_{D'}(T'(\lambda))\) is strictly increasing as a function of \(T'(\lambda)\) on  \([2\sqrt{D'},\infty)\)  (see 
Lemma~\ref{lem:Phi-monotone}), we have that \(z'=0\) (otherwise, \(T'(\lambda)\) is not a constant; hence \(\Phi_{D'}(T'(\lambda))\) would  not be  constant).
Hence, 
\[
   z=z'=0.
\]

Now suppose \(z\ne0\).
We claim that \(z'\ne 0\) (otherwise, by the same argument as above, we have \(z=0\)).
Since
\(\Phi_{D}(T(\lambda))\)
is increasing as a function of $T(\lambda)$ on  \( [2\sqrt{D},\infty)\) (see 
Lemma \ref{lem:Phi-monotone}), by \eqref{T lambda}, the function
\[
   \lambda\mapsto \Phi_{D}(T(\lambda)) = \|\lambda e-v\|_p^p
\]
attains its minimum when \(T(\lambda)=2-2\operatorname{Re}(\bar\lambda z)\) is minimal, which is equivalent to that
\[
   \operatorname{Re}(\bar\lambda z)
\]
is maximal, namely at the unique point
\[
   \lambda=\frac{z}{|z|}.
\]
Similarly,
\[
   \lambda\mapsto \Phi_{D'}(T'(\lambda)) =\|\lambda e'-v'\|_p^p
\]
attains its minimum at the unique point
\[
   \lambda'=\frac{z'}{|z'|}.
\]
By \eqref{Phi,trace}, their minimum points coincide. Hence, we have that
\[
   \frac{z}{|z|}
   =
   \frac{z'}{|z'|}.
\]
Therefore, there exist \(\rho,\rho'>0\) and \(\omega\in\mathbb T\) such that
\begin{align}\label{rho}
   z=\rho\omega,
   \qquad
   z'=\rho'\omega.
\end{align}

Taking \(\lambda=\omega e^{i\theta}\), we have that
\[
   \operatorname{Re}(\bar\lambda z)=\rho\cos\theta,
   \quad
   \operatorname{Re}(\bar\lambda z')=\rho'\cos\theta.
\]
Therefore, 
\begin{align}\label{cos}
   \Phi_D(2-2\rho\cos\theta)
   \stackrel{\eqref{T lambda},\eqref{Phi,trace}}{=}
   \Phi_{D'}(2-2\rho'\cos\theta),
   \qquad \theta\in [0,2\pi).
\end{align}

{\bf The case when $p=2$. }
If \(p=2\), then
\[
   \|\lambda e-v\|_2^2
   =
   {\rm Tr}\bigl((\lambda e-v)^*(\lambda e-v)\bigr)
   =
   2-2\operatorname{Re}(\bar\lambda z),
\]
and similarly,
\[
   \|\lambda e'-v'\|_2^2
   =
   2-2\operatorname{Re}(\bar\lambda z').
\]
Then we have that
\[
   2-2\operatorname{Re}(\bar\lambda z)
   \stackrel{\eqref{e,e'}}{=}
   2-2\operatorname{Re}(\bar\lambda z'),
   \qquad \lambda\in\mathbb T.
\]
Taking \(\lambda=\omega\), by \eqref{rho}, we obtain
\[
   2-2\rho=2-2\rho',
\]
and hence
\[
   \rho=\rho'.
\]
Therefore
\[
   z=\rho\omega=\rho'\omega=z'.
\]
That is,
\[
   {\rm Tr}(e^*v)\stackrel{\eqref{z,z'}}{=}{\rm Tr}((e')^*v').
\]

{\bf The case when $p\ne 2$. }  Taking \(\theta=\pi/2\), by \eqref{cos}, we obtain
\[
   \Phi_D(2)=\Phi_{D'}(2),
\]
which implies that
\begin{align}\label{D,D'}
   (1+\sqrt{1-D})^{\frac{p}{2}}
   +
   (1-\sqrt{1-D})^{\frac{p}{2}}
   =
   (1+\sqrt{1-D'})^{\frac{p}{2}}
   +
   (1-\sqrt{1-D'})^{\frac{p}{2}}.
\end{align}
Let \(r=p/2\). The function
\[
   u\mapsto (1+u)^r+(1-u)^r,
   \qquad 0\le u\le1,
\]
is strictly monotone whenever \(r\ne1\). Indeed,
\[
   \frac{d}{du}\bigl((1+u)^r+(1-u)^r\bigr)
   =
   r\bigl((1+u)^{r-1}-(1-u)^{r-1}\bigr),
\]
which is strictly positive for \(r>1\) and strictly negative for \(0<r<1\) whenever
\(0\le u<1\). Hence, by \eqref{D,D'}, we have 
\[
   D=D'.
\]
Hence, 
\[
   \Phi_D(2-2\rho\cos\theta)
   \stackrel{\eqref{cos}}{=}
   \Phi_D(2-2\rho'\cos\theta),
   \qquad \theta\in [0,2\pi).
\]
Since \(\Phi_D\) is strictly increasing, we have that
\[
   2-2\rho\cos\theta
   =
   2-2\rho'\cos\theta,
   \qquad \theta\in [0,2\pi).
\]
Taking \(\theta=0\), we obtain
\[
   \rho=\rho'.
\]
Therefore, 
\[
   z=\rho\omega=\rho'\omega=z'.
\]
That is,
\[
   {\rm Tr}(e^*v)\stackrel{\eqref{z,z'}}{=}{\rm Tr}((e')^*v').
\]
The proof is complete.
\end{proof}

Proposition \ref{rank-one decide+} below provides a criterion for the coincidence of  two positive operators in terms of $\norm{\cdot}_p$. Note that the case when $p>1$
  was treated in \cite[Lemma]{Nagy}. However, the techniques used in \cite{Nagy} are not applicable in the setting when $p\le 1$ due to the fact that  $\norm{\cdot}_p$ is not monotone with  respect to submajorization (Ky Fan norms).
We need the following lemma before proving Proposition \ref{rank-one decide+}.

\begin{lemma}\label{minimizers-rank-one}
Let $H$ be a finite-dimensional complex Hilbert space, let $0<p<\infty$ and $\gamma\ge 1$. For $T\in S(C_p(H)^+)$, define 
\[f_T:Proj_1(H)\to \mathbb R,\quad
f_T(e):=\left\|T-\gamma e\right\|_p^p.\]
Let \(\lambda_1\ge \lambda_2\ge \cdots \ge \lambda_n\)
be the eigenvalues of $T$, and let
\(M_T:=\ker\bigl(T-\lambda_1 I\bigr)\)
be the eigenspace corresponding to the maximal eigenvalue \(\lambda_1\). Then, 
\[
\min_{e\in Proj_1(H)} f_T(e)
=
\bigl(\gamma-\lambda_1\bigr)^p+1-\lambda_1^p.
\]
Moreover, for $e\in Proj_1(H)$, \(
f_T(e)\) attains its minimum value if and only if \(
\operatorname{Ran}(e)\subset M_T.
\)
\end{lemma}

\begin{proof}
    Since \(H\) is finite-dimensional, the set \(Proj_1(H)\) is compact. Moreover, the map \(f_T\)
is continuous. Therefore, \(f_T\) attains its minimum value on \(Proj_1(H)\). 

Fix an arbitrary $e\in Proj_1(H)$. Let $\xi \in \operatorname{Ran}(e)$
be a unit vector. Set
\[
T':=T-\gamma e,\qquad Q:=I-e,\qquad H':=\xi^\perp=\operatorname{Ran}(Q).
\]
Let \(\mu_1\ge \mu_2\ge \cdots \ge \mu_n\)
be the eigenvalues of the self-adjoint operator $T'$. Then
\[
f_T(e)=\|T'\|_p^p=\sum_{j=1}^n |\mu_j|^p.
\]
Let \(B:=(QTQ)|_{H'}\). Since   $Qz=z$ for every $z\in H'$, it follows that
\[
\langle Bz,z\rangle=\langle QTQz,z\rangle=\langle Tz,z\rangle\ge 0,
\]
which implies that $B\ge 0$. Let
\(\eta_1\ge \eta_2\ge \cdots \ge \eta_{n-1}\ge 0\) be the eigenvalues of $B$. We observe that $B$ is the principal submatrix of $T'$ corresponding to the decomposition
$H=\mathbb C\xi \oplus H'$, so Cauchy's interlacing theorem \cite[Theorem 1]{Cauchy} (see also \cite[Theorem 4.3.17]{Horn}) yields
\[
\mu_1\ge \eta_1\ge \mu_2\ge \eta_2\ge \cdots \ge \eta_{n-1}\ge \mu_n,
\]
which implies that
\begin{align}\label{mu-eta}
    \mu_j\ge \eta_j>0,\quad j=1,\dots,n-1.
\end{align}
On the other hand, $B$ is also the principal submatrix of $T$ corresponding to
the same decomposition, and another application of Cauchy's interlacing theorem \cite[Theorem 1]{Cauchy} yields
\[
\lambda_1\ge \eta_1\ge \lambda_2\ge \eta_2\ge \cdots
\ge \eta_{n-1}\ge \lambda_n\ge 0,
\]
which implies that 
\begin{align}\label{eta-lambda}
    \eta_j\ge \lambda_{j+1}\ge 0,\quad j=1,\dots,n-1
\end{align}
Then we get
\begin{align}\label{mu-lambda}
\mu_j\stackrel{\eqref{mu-eta},\eqref{eta-lambda}}{\ge} \lambda_{j+1}\ge 0,\quad j=1,\dots,n-1.
\end{align}
Next, since \(T'\) is self-adjoint, for the unit vector \(\xi \in \operatorname{Ran}(e)\), by \cite[Theorem 4.2.2(c)]{Horn}, we have
\begin{align}\label{mu_n}
\mu_n\le \langle T'\xi,\xi\rangle
=
\langle T \xi,\xi\rangle-\gamma.
\end{align}
Since $T$ is self-adjoint, it follows from \cite[Theorem 4.2.2(c)]{Horn} that 
\[\langle T\xi,\xi\rangle\le \lambda_1\le 1\le \gamma.\]
By  \eqref{mu_n}, we have 
\begin{align}\label{mu1}
|\mu_n|=-\mu_n\ge \gamma-\langle T\xi,\xi\rangle\ge \gamma-\lambda_1.
\end{align}
Let \((s_1,s_2,\dots,s_n)\) be the singular values of $T'$.
Since \(T'\) is self-adjoint, 
the singular values of $T'$ satisfy
\begin{align}\label{mu(T')}
(s_1,s_2,\dots,s_n)&\quad =~s(|\mu_1|,|\mu_2|,\dots,|\mu_n|) = s(\mu_1,\mu_2,\dots,\mu_{n-1},-\mu_n)
\\&\stackrel{\eqref{mu-lambda},\eqref{mu1}}{\ge} s(\lambda_2,\dots,\lambda_n,\gamma-\lambda_1).\nonumber
\end{align}
We observe that
if, in addition,
 $e$ is such that \(\operatorname{Ran}(e)\subset M_T\), then equality in \eqref{mu(T')} holds and hence
\(f_T(e)\) attains its minimum value by the unitary invariance of \(\norm{\cdot}_p\). Thus,
\[\min_{e\in Proj_1(H)} \left\|T-\gamma e\right\|_p^p\stackrel{\eqref{mu(T')}}{=}(\gamma-\lambda_1)^p+1-\lambda_1^p.\]

Conversely, if \(e\in Proj_1(H)\) is such that \(f_T(e)\) attains its minimum value, then
\begin{align*}
(s_1,s_2,\dots,s_n)=s(\mu_1,\mu_2,\dots,\mu_{n-1},-\mu_n)\stackrel{\eqref{mu(T')}}{=} 
s(\lambda_2,\dots,\lambda_n,\gamma-\lambda_1),
\end{align*}
which implies that the inequalities \eqref{mu-lambda} and \eqref{mu1} must be equalities. Hence, for the unit vector $\xi \in \operatorname{Ran}(e)$,
we have that
\[
-\mu_n\stackrel{\eqref{mu1}}{=} \gamma-\langle T\xi,\xi\rangle\stackrel{\eqref{mu1}}{=}\gamma-\lambda_1,
\]
which implies that
\[
\langle T\xi,\xi\rangle=\lambda_1.
\]
Since \(\lambda_1\) is the maximum eigenvalue of $T$, by \cite[Theorem 4.2.2(c)]{Horn}, we have that
\[T\xi=\lambda_1 \xi.\]
Therefore, we have $\xi\in M_T$, i.e.
\[
\operatorname{Ran}(e)\subset M_T.
\]
This completes the proof.\end{proof}

\begin{proposition}\label{rank-one decide+}
Let $H$ be a finite-dimensional complex Hilbert space, and let $0<p<\infty,\gamma\ge 1$. If \(x,y\in S(C_p(H)^+)\) satisfy the equality
\begin{align}\label{x-gamme e}
    \left\|x-\gamma e\right\|_p^p=\left\|y-\gamma e\right\|_p^p,
\end{align}
for all \(e\in Proj_1(H)\), then \(x=y\).
\end{proposition}
\begin{proof}
Arguing similarly to the proof of \cite[Lemma]{Nagy}, we use induction on \(\dim(H)=n\). It is obvious that the lemma holds for \(n=1\). Assume it holds for all dimensions less than \(n\). Let \(\lambda_1\ge \lambda_2\ge \cdots \ge \lambda_n\), \(\mu_1\ge \mu_2\ge \cdots \ge \mu_n\) be the eigenvalues of \(x,y\), respectively, and let \(M_x:=\ker\bigl(x-\lambda_1 I\bigr)\), \(M_y:=\ker\bigl(y-\lambda_1 I\bigr)\). 
It follows from \eqref{x-gamme e} that 
\[\min_{e\in Proj_1(H)} \left\|x-\gamma e\right\|_p^p
=\min_{e\in Proj_1(H)} \left\|y-\gamma e\right\|_p^p.\]
By Lemma \ref{minimizers-rank-one}, we have
\[(\gamma-\lambda_1)^p+1-\lambda_1^p=(\gamma-\mu_1)^p+1-\mu_1^p.\]
Since the function \(t\mapsto (\gamma-t)^p+1-t^p\) is strictly decreasing on \([0,1]\), we have \(\lambda_1=\mu_1\). Moreover, by \eqref{x-gamme e} and Lemma \ref{minimizers-rank-one}, we have \(M_x=M_y\). Hence, 
\begin{align}\label{Mx}
x|_{M_x}=y|_{M_y}.
\end{align}
If \(\|x|_{M_x}\|_p=\|y|_{M_y}\|_p=1\), the proof is complete. Assume now that \(\|x|_{M_x}\|_p=\|y|_{M_y}\|_p<1\) and
consider the operators \(x|_{M_x^\perp}, y|_{M_y^\perp}\) and the subspace \(M_x^\perp\) with \(\dim(M_x^\perp)<n\). We observe that 
\begin{align}\label{M_x perp}
1>  \left\|x|_{M_x^\perp}\right\|_p^p=\left\|x\right\|_p^p-\left\|x|_{M_x}\right\|_p^p
\stackrel{\eqref{Mx}}{=}\left\|y\right\|_p^p-\left\|y|_{M_y}\right\|_p^p=\left\|y|_{M_y^\perp}\right\|_p^p> 0.
\end{align}
For every \(e\in Proj_1(M_x^\perp)\), we have that
\begin{align}\label{e in M_x perp}
\left\|x|_{M_x^\perp}-\gamma e\right\|_p^p&=\left\|x-\gamma e\right\|_p^p-\left\|x|_{M_x}\right\|_p^p
\stackrel{\eqref{x-gamme e},\eqref{Mx}}{=}\left\|y-\gamma e\right\|_p^p-\left\|y|_{M_y}\right\|_p^p
\\&=\left\|y|_{M_y^\perp}-\gamma e\right\|_p^p \nonumber.
\end{align}
Let 
\[\gamma':=\frac {\gamma}{\left\|x|_{M_x^\perp}\right\|_p}\stackrel{\eqref{M_x perp}}{>} 1.\]
Then, for every \(e\in Proj_1(M_x^\perp)\),
\[\left\|\frac{x|_{M_x^\perp}}{\left\|x|_{M_x^\perp}\right\|_p}-\gamma' e\right\|_p^p
\stackrel{\eqref{e in M_x perp}}{=}\left\|\frac{y|_{M_y^\perp}}{\left\|y|_{M_y^\perp}\right\|_p}-\gamma' e\right\|_p^p.\]
By the induction hypothesis, \(x|_{M_x^\perp}=y|_{M_y^\perp}\). Hence, \(x=y\).
\end{proof}

A useful criterion for the coincidence of  two (not necessarily positive) operators in $C_p(H)$,  \(p>1\),  is given in \cite[Proposition 2.10]{Polo}.
Proposition \ref{rank-one partial isometry decide} below extends 
\cite[Proposition 2.10]{Polo} to the setting \(0<p\ne 2<\infty\). 
Before proceeding to the proof of 
Proposition~\ref{rank-one partial isometry decide}, we need the following lemma, whose proof relies on Weyl submajorization (logarithmic submajorization). 
\begin{lemma}\label{positive decide}
Let \(H\) be a finite-dimensional Hilbert space, and let
\(a\in C_p(H)\) be an invertible positive operator  (for simplicity, we write \(a>0\)), where \(0<p<\infty\). For every \(e\in Proj_1(H)\) and every \(\lambda\in\mathbb T\), we have
\begin{align}\label{lambda0}
   \|a-e\|_p^p\le \|a-\lambda e\|_p^p ,
\end{align}
and  equality holds if and only if \(\lambda =1\).

Conversely, if \(a\in C_p(H)\) satisfies the preceding conditions, then \(a>0\).
\end{lemma}
\begin{proof}
   Assume that \(\dim(H)=n\).

   {\bf{Step 1}} We prove the inequality \eqref{lambda0} when \(a>0\). 
   
   For a fixed \(e\in Proj_1(H)\), let \(x_1\ge x_2\ge \cdots\ge x_n\) be the singular values of \(a-e\), and let \(y_1\ge y_2\ge \cdots\ge y_n\) be the singular values of \(a-\lambda e\) with \(\lambda\in\mathbb T\). Let
   \[s(a-e)=(x_1, x_2, \ldots, x_n),\quad s(a-\lambda e)=(y_1, y_2, \ldots, y_n).\]
   Since \(a,e\ge 0\) and \(\lambda\in\mathbb T\), it follows from  \cite[Theorem 2]{zhan} that 
   \[\prod_{i=1}^{k}x_i\le \prod_{i=1}^{k}y_i, \quad 1\le k\le n.\]
   It follows from \cite[Proposition 1.3]{Hiai} that
   \[\sum_{i=1}^{k}x_i^p\le \sum_{i=1}^{k}y_i^p, \quad 1\le k\le n.\]
   Hence, we have that
   \[
   \|a-e\|_p^p\le \|a-\lambda e\|_p^p.
   \]
   {\bf{Step 2}} We prove the strict inequality when \(a>0\) and \(\lambda\in\mathbb T\setminus \{1\}\).

   Let
   \[
   P=P_\xi=\xi \otimes \xi,\quad \|\xi\|=1.
   \]
   Since \(a>0\) is invertible, \(a^{-1}>0\). We have that
   \[
   \theta=\langle a^{-1}\xi,\xi\rangle>0 .
   \]
   Since \(a\) is invertible and \(P\) is a rank-one projection, by the rank-one determinant formula (see \cite[Lemma 1.1]{Ding.J}), we have
   \begin{align}\label{rank-one}
    \det(a-\lambda P)
   =
   \det(a)\bigl(1-\lambda\langle a^{-1}\xi,\xi\rangle\bigr)
   =\det(a)(1-\lambda\theta).
   \end{align}
   In particular,
   \[
   \det(a-P)=\det(a)(1-\theta).
   \]
   It is known that
   \begin{align}\label{det}
   \prod_{i=1}^{n}x_i=\left|\det(a- e)\right|,\quad \prod_{i=1}^{n}y_i=\left|\det(a-\lambda e)\right|.
   \end{align}
   Hence, we have that
   \begin{align}\label{positive =}
   \prod_{i=1}^{n}y_i-\prod_{i=1}^{n}x_i&\stackrel{\eqref{det}}{=}|\det(a-\lambda e)|-|\det(a-e)|
   \\&\stackrel{\eqref{rank-one}}{=}\det(a)(|1-\lambda\theta|-|1-\theta|).\nonumber
   \end{align}
   Since \(\lambda\in\mathbb T\setminus \{1\}\), we have that
   \begin{align}\label{lambda 1}
   |1-\lambda\theta|^2-|1-\theta|^2
   &=(1+\theta^2-2\theta\operatorname{Re}\lambda)-(1+\theta^2-2\theta)
   \\&=2\theta(1-\operatorname{Re}\lambda)>0.\nonumber
   \end{align}
   Then we have that
   \begin{align}\label{positive >}
       \prod_{i=1}^{n}y_i-\prod_{i=1}^{n}x_i\stackrel{\eqref{positive =},\eqref{lambda 1}}{>}0
   \end{align}
   Assume that there exists \(\lambda\in\mathbb T\setminus \{1\}\) such that
   \[
   \|a-e\|_p^p= \|a-\lambda e\|_p^p,
   \]
   that is,
   \[
   \sum_{i=1}^{n} x_i^p=\sum_{i=1}^{n} y_i^p.
   \]
   Since for every \(\varepsilon>0\), the function \(t\mapsto -\log(t+\varepsilon)\) is convex, by \cite[Proposition 1.1]{Hiai}, we have that
   \[
   -\sum_{i=1}^{n} \log(x_i^p+\varepsilon)
   \le
   -\sum_{i=1}^{n} \log(y_i^p+\varepsilon),
   \]
   which implies that
   \[
   \sum_{i=1}^{n} \log(x_i^p+\varepsilon)=\log \left( \prod_{i=1}^{n}(x_i^p+\varepsilon)\right)\ge \sum_{i=1}^{n} \log(y_i^p+\varepsilon)=\log \left( \prod_{i=1}^{n}(y_i^p+\varepsilon)\right).
   \]
   Letting \(\varepsilon\downarrow 0\), we obtain that
   \[
   \left(\prod_{i=1}^{n} x_i\right)^p\ge \left(\prod_{i=1}^{n} y_i\right)^p.
   \]
   Since the function \(t\mapsto t^p\) is increasing, we have
   \[
   \prod_{i=1}^{n} x_i\ge \prod_{i=1}^{n} y_i,
   \]
   which contradicts the inequality \eqref{positive >}. Therefore, we have
   \begin{align*}
    \|a-e\|_p^p< \|a-\lambda e\|_p^p,
   \end{align*}
   for \(e\in Proj_1(H)\) and \(\lambda\in\mathbb T\setminus \{1\}\).

   {\bf{Step 3}} Let  \(a\in C_p(H)\) satisfy \eqref{lambda0} and equality holds if and only if \(\lambda =1\).
   We prove that \(a>0\).

   We first prove that \(r(a)=I\). Otherwise, there exists \(\xi\in H\) such that
   \[
   a\xi=0.
   \]
   Letting \(e=\xi \otimes \xi\), we have that
   \[
   ae=0=ea^*.
   \]
   Hence, for \(\lambda\in\mathbb T\), we have
   \begin{align*}
   (a-\lambda e)(a-\lambda e)^*&=aa^*-\bar{\lambda}ae^*-\lambda ea^*+\lambda\bar{\lambda}ee^*
   \\&=aa^*-\bar{\lambda}a e-\lambda ea^*+\lambda\bar{\lambda}e
   \\&=aa^*+e,
   \end{align*}
   which implies that 
   \[
   \left\|a-\lambda e\right\|_p^p=\left\|(a-\lambda e)^*\right\|_p^p=\left\|\left((a-\lambda e)(a-\lambda e)^*\right)^{\frac{1}{2}}\right\|_p^p
   \]
  does not depend on \(\lambda\), contradicting the condition that the 
  equality holds if and only if \(\lambda =1\)
  in the lemma.

  Next, we prove that \(a>0\). By \cite[Theorem 1.7.3]{DPS}, there exists a unitary operator \(u\in \mathcal{B}(H)\) such that
   \[
   a=u|a|,
   \]
in other words, 
   \[
   |a|=u^*a.
   \]
   Since \(r(a)=I\), it follows that  \(|a|>0\). By the above argument,   \(|a|\) satisfies the conditions. 
   Since \(\|\cdot\|_p^p\) is unitarily invariant, we have  
   \begin{align}\label{unitary}
   \left\|a-\lambda e\right\|_p^p=\left\||a|-\lambda u^* e\right\|_p^p.
   \end{align}
   By \cite[Theorem 2.5.3]{Horn}, there exists an
   orthonormal basis \(\{\eta_1,\ldots,\eta_n\}\) of \(H\) consisting of eigenvectors
   of \(u^*\), and numbers \(\{\omega_1,\ldots,\omega_n\}\subset \mathbb{T}\) such that
   \[
   u^*\eta_i=\omega_i\eta_i,\qquad i=1,\ldots,n.
   \]
   If \(u^*\ne I\), there exist \(\omega_0\in \{\omega_1,\ldots,\omega_n\}\subset \mathbb{T}\) and \(\eta_0\in \{\eta_1,\ldots,\eta_n\}\) such that
   \begin{align}\label{omege0,eta0}
   \omega_0\ne 1,\quad u^*\eta_0=\omega_0\eta_0
   \end{align}
   Letting \(e=\eta_0\otimes\eta_0\), we have that 
   \begin{align*}
   \left\||a|-\omega_0 e\right\|_p^p&\stackrel{\eqref{omege0,eta0}}{=}\left\||a|- u^* e\right\|_p^p\stackrel{\eqref{unitary}}{=}\left\|a- e\right\|_p^p
   \\&\stackrel{\eqref{lambda0}}{<}\left\|a- \lambda e\right\|_p^p
   \stackrel{\eqref{unitary}}{=}\left\||a|-\lambda u^* e\right\|_p^p\stackrel{\eqref{omege0,eta0}}{=}\left\||a|-\lambda \omega_0 e\right\|_p^p,
   \end{align*}
   which implies that
   \begin{align}\label{omega0}
   \left\||a|-\omega_0 e\right\|_p^p<\left\||a|-\lambda \omega_0 e\right\|_p^p.
   \end{align}
   Let 
   \[\lambda_0=\bar{\omega}_0\stackrel{\eqref{omege0,eta0}}{\ne} 1.\]
   Then we have
   \[
   \left\||a|-\omega_0 e\right\|_p^p\stackrel{\eqref{omega0}}{<}\left\||a|-\lambda_0 \omega_0 e\right\|_p^p=\left\||a|- e\right\|_p^p,
   \]
   which contradicts \eqref{lambda0}. Hence, we have  
   \[u^*=I,\]
   which implies that
   \[
   a=|a|>0.
   \]
   This completes the proof.
\end{proof}

\begin{proposition}\label{rank-one partial isometry decide}
    Let $H$ be a finite-dimensional complex Hilbert space, let $0<p\neq 2<\infty$. If \(x,y\in S(C_p(H))\) satisfy
    \[\left\|x- e\right\|_p^p=\left\|y- e\right\|_p^p, \]
    for all \(e\in \mathcal U_{\min}(H)\), then \(x=y\).
\end{proposition}
\begin{proof}
    {\bf{Step 1}} For \(x\in S(C_p(H))\), by \cite[Theorem 1.7.3]{DPS}, there exists a unitary \(u\in \mathcal{B}(H)\) such that
    \[x=u^*|x|,\qquad uu^*= r(x), \qquad u^*u= \ell(x).\] 
    In other words,
    \[|x|=ux.\]
    Since \(\left\|\cdot\right\|_p^p\) is unitarily invariant, we have  
    \[\left\||x|- ue\right\|_p^p=\left\|ux- ue\right\|_p^p=\left\|uy- ue\right\|_p^p, \]
    for all \(e\in \mathcal U_{\min}(H)\). Since \(ue\) is also a minimal partial isometry, we can write
    \begin{align}\label{condition}
    \left\||x|- e\right\|_p^p=\left\|uy- e\right\|_p^p, 
    \end{align}
    for all \(e\in \mathcal U_{\min}(H)\).

    {\bf{Step 2}} We prove that \( \ell(uy)= r(uy)= r(|x|)\). 
    
    If \(|x|>0\) on \(H\), then,
    by \eqref{condition} and Lemma \ref{positive decide}, \(uy\) satisfies the condition in Lemma~\ref{positive decide}. 
    Then, \(uy>0\), which implies that 
    \[ \ell(uy)= r(uy)=I= r(|x|).\] 
    If \(|x|\) is not invertible,
    we may assume that \( r(uy)\neq  r(|x|)\). At least one of the following inequalities holds:
    \[ r(uy) (I-  r(|x|)) \ne 0
        \text{ or }  r(|x|)(I- r(uy)) \ne 0.\]
    Assume the former one holds, and let
    \[e_0\le  I-  r(|x|))   \]
    be a rank-one projection  such that $r(uy) e_0\ne 0.$
    Hence, we have 
    \begin{align}\label{e0}
    e_0\perp |x|, \quad e_0\not\perp uy.
    \end{align}
We have 
    \[\left\|uy- e_0\right\|_p^p\stackrel{\eqref{condition}}{=}\left\||x|- e_0\right\|_p^p=2,\quad
    \left\|uy+ e_0\right\|_p^p\stackrel{\eqref{condition}}{=}\left\||x|+ e_0\right\|_p^p=2.
    \]
    By  \cite[Theorem A.1]{RX}, we have  
    \[uy\perp e_0\]
    which contradicts \eqref{e0}.
    Hence,  
    \[ r(uy)=  r(|x|).\]
    The same argument yields that \( \ell(uy)=  r(|x|)\). 

    {\bf{Step 3}} We consider the subspace \(\cB(H_0)= r(|x|)\cB(H) r(|x|)\). By Step 2, we have 
    \[|x|,uy\in S(C_p(H_0)).\]
    Then, 
    \begin{align}\label{condition2}
    \left\||x|- \lambda q\right\|_p^p\stackrel{\eqref{condition}}{=}\left\|uy-\lambda q\right\|_p^p
    \end{align}
    for \(q\in Proj_1(H_0)\) and \(\lambda\in \mathbb T\). Since \(|x|>0\) on \(H_0\), it follows from Lemma \ref{positive decide} that, for every \(q\in Proj_1(H_0)\) and every \(\lambda\in\mathbb T\), we have 
\[
   \||x|-q\|_p^p\le \||x|-\lambda q\|_p^p ,
\]
and equality holds if and only if \(\lambda =1\). Then, by \eqref{condition2}, \(uy\) satisfies the conditions in Lemma \ref{positive decide}. Hence,
\[uy>0\]
on \(H_0\). 
Then, by \eqref{condition}, \(|x|,uy\) also satisfy the condition in Lemma \ref{rank-one decide+} with \(\gamma=1\). Hence,  
\[uy=|x|\]
on \(H_0\). Since \(uy,|x|\in S(C_p(H_0))\), we have
\[uy=|x|\]
which implies that
\[x=y.\]
The proof is complete. 
\end{proof}

We recall the following result concerning the general form of a bijective function on the minimal partial isometries \cite[Theorem 2]{Molnar} (see also \cite[Theorem 2.15]{Polo})
as  the last preparation for proving Theorem \ref{main}.

\begin{theorem}\label{wigner}
    Let $H$ be a complex Hilbert space. Let \(F:\cU_{\min}(H)\to\cU_{\min}(H)\) be a bijective function with the property that
    \[
    {\rm Tr}(F(e)^*F(v))={\rm Tr}(e^*v),\quad e,v\in \cU_{\min}(H).
    \]
    Then \(F\) is of one of the following forms:
    \begin{enumerate}[label=(\roman*)]
        \item there exist unitaries \(U,V\) on \(H\) such that
        \[F(e)=UeV,\quad e\in\cU_{\min}(H).\]
        \item there exist antiunitaries \(U,V\) on \(H\) such that
        \[F(e)=Ue^*V,\quad e\in\cU_{\min}(H).\]
    \end{enumerate}
\end{theorem}

Now, we are ready to present  the proof of Theorem \ref{main}.
\begin{proof}[Proof of Theorem \ref{main}]
    By Lemma \ref{orth-preserving}, we have \(\dim(H_1)=\dim(H_2)\). Hence, we can identify \(H_1\) with \(H_2\) and denote it by \(H\). 
    By Proposition \ref{U_min preserving}, we have \(V_0(\cU_{\min}(H))= \cU_{\min}(H)\) and 
    we first assume that
    \begin{align}\label{main lambda}
    V_0(\lambda e)=\lambda e, 
    \end{align}
    for all \(e\in \cU_{\min}(H)\) and \(\lambda\in \mathbb T\). The other case will be proved later.
    Then we have 
    \[
    \left\|e-\lambda f\right\|_p=\left\|V_0(e)-V_0(\lambda f)\right\|_p\stackrel{\eqref{main lambda}}{=}\left\|V_0(e)-\lambda V_0(f)\right\|_p,
    \]
    for arbitrary \(e,f\in \cU_{\min}(H)\). 
    By Lemma \ref{lem:trace-recovery}, we have
    \[
    {\rm Tr}(e^*f)={\rm Tr}(V_0(e)^*V_0(f)).
    \]
    Thus \(V_0:\cU_{\min}(H)\to \cU_{\min}(H)\) satisfies the property in Theorem \ref{wigner}. Hence, \(V_0:\cU_{\min}(H)\to \cU_{\min}(H)\) is of one of the following forms:
    \begin{enumerate}[label=(\roman*)]
        \item there exist unitaries \(U,V\) on \(H\) such that
        \[V_0(e)=UeV,\quad e\in\cU_{\min}(H).\]
        \item there exist antiunitaries \(U,V\) on \(H\) such that
        \[V_0(e)=Ue^*V,\quad e\in\cU_{\min}(H).\]
    \end{enumerate}
    Assume that form (i) holds; form (ii) can be proved by a similar argument.
    Define a surjective isometry \(\widetilde{V_0}:S(C_p(H))\to S(C_p(H))\) by
    \[\widetilde{V_0}(x)=U^*V_0(x)V^*\]
    for \(x\in S(C_p(H))\). Then we have
    \[
    \widetilde{V_0}(e)=e
    \]
    for all \(e\in\cU_{\min}(H)\). Hence,
    \begin{align}\label{theorem=}
    \|x-e\|_p^p=\|\widetilde{V_0}(x)-\widetilde{V_0}(e)\|_p^p=\|\widetilde{V_0}(x)-e\|_p^p
    \end{align}
    for arbitrary \(x\in S(C_p(H))\) and \(e\in \cU_{\min}(H)\). In particular, \eqref{theorem=} holds for all finite-rank operators \(x\in S(C_p(H))\). 
    As in Step 2 of Proposition \ref{rank-one partial isometry decide}, \(\widetilde{V_0}(x)\) is also a finite-rank operator.
    Hence, by Proposition \ref{rank-one partial isometry decide},
    \[\widetilde{V_0}(x)=x\]
    for all finite-rank operators \(x\in S(C_p(H))\). Since the set of norm $1$ finite-rank operators is dense in \(S(C_p(H))\) and \(\widetilde{V_0}\) is continuous, \(\widetilde{V_0}(x)=x\) for all \(x\in S(C_p(H))\). Therefore, \(V_0(x)=UxV\) for all \(x\in S(C_p(H))\).
    In particular, $V_0$ can be extended to a surjective complex-linear isometry on $C_p(H)$. 

    If
    \[
    V_0(\lambda e)=\bar{\lambda} e, 
    \]
    for all \(e\in \cU_{\min}(H)\) and \(\lambda\in \mathbb T\), then we consider the surjective isometry \(\overline{V_0}:S(C_p(H))\to S(C_p(H))\) defined by
    \[\overline{V_0}(x)=\overline {V_0(x)}.\]
    The result then follows by the same argument as above.
\end{proof}


\section{Remarks on isometry extension from the positive sphere of $C_p$}\label{s:p}

Nagy studied   Tingley's problem for \(S(C_p(H)^+)\) in \cite{Nagy} with  \(p>1\).
Note that the case when $p=2$ is covered in \cite{Nagy}. 
In this section, we extend his result to the setting  of \(p>0\).
\begin{lemma}\label{tr-preserving for Proj}
    Let \(e,f,e',f'\in Proj_1(H)\), and let \(p>0\). If 
    \begin{align}\label{e-f,proj}
        \left\|e-f\right\|_p^p=\left\|e'-f'\right\|_p^p,
    \end{align}
    then
    \[{\rm Tr}(ef)={\rm Tr}(e'f').\]
\end{lemma}
\begin{proof}
    Let \(h=e-f\). Since \(e,f\in Proj_1(H)\), it follows that \(h\) is a self-adjoint operator with rank at most 2. Hence, \(h\) has at most two nonzero eigenvalues. Let \(\lambda_1, \lambda_2\) be the nonzero eigenvalues of \(h\). By the additivity of \({\rm Tr}\), we have
    \[\lambda_1+\lambda_2={\rm Tr}(h)=0,\]
    which implies \(\lambda_1=-\lambda_2\). Hence,
\begin{align}\label{Tr=0}
    \left\|h\right\|_p^p&=\left| \lambda_1 \right|^p + \left| \lambda_2 \right|^p = 2\left| \lambda_1 \right|^p.
\end{align}
Since 
\begin{align}\label{Tr1}
    {\rm Tr}(h^2)&={\rm Tr}((e-f)^2)={\rm Tr}(e)+{\rm Tr}(f)-{\rm Tr}(ef)-{\rm Tr}(fe)
    \\&={\rm Tr}(e)+{\rm Tr}(f)-2{\rm Tr}(ef)=2-2{\rm Tr}(ef),\nonumber
\end{align}
and 
\begin{align}\label{Tr2}
    {\rm Tr}(h^2)=\lambda_1^2+\lambda_2^2=2\lambda_1^2,
\end{align} 
it follows that 
 \[\left\|e-f\right\|_p^p\stackrel{\eqref{Tr=0}}{=} 2\lambda_1^2\stackrel{\eqref{Tr2}}{=}2(1-{\rm Tr}(ef))^{\frac{p}{2}}.\]
 Similarly,
 \[\left\|e'-f'\right\|_p^p=2(1-{\rm Tr}(e'f'))^{\frac{p}{2}}.\]
 Since the function \(t\mapsto (1-t)^{p/2}\) is strictly decreasing on \([0,1]\) and \({\rm Tr}(ef),{\rm Tr}(e'f')\in [0,1]\), it follows from  \eqref{e-f,proj} that 
 \[{\rm Tr}(ef)={\rm Tr}(e'f').\]
 The proof is complete. 
\end{proof}

\begin{proposition}
    
\label{positive sphere}
    Let \(H_1,H_2\) be complex Hilbert spaces, and let \(V_0: S(C_p(H_1)^+) \to S(C_p(H_2)^+)\) be a surjective isometry, where \(0<p<\infty\). Then there exists a surjective real-linear isometry that extends \(V_0\). 
\end{proposition}
\begin{proof}
    We prove the statement in  steps similar to those in the proof of Theorem~\ref{main}.
    
    {\bf{Step 1}} We first prove that \(V_0\) preserves orthogonality.

    For \(0<p<1\), since Lemma \ref{orth for p<1} holds for \(C_p(H)^+\) with \(0< p< 1\), the proof is similar to the proof of Lemma \ref{orth-preserving}.

    For \(p=1\), this was  proved in \cite[Lemma 2.2]{Polo2}.
    
    For \(p>1\), this was proved in \cite[p.~7]{Nagy}.

  By the same argument as in the proof of Lemma \ref{orth-preserving}, 
    we have \(\dim(H_1)=\dim(H_2)\). Hence, we can identify \(H_1\) with \(H_2\) and denote it by \(H\). 

    {\bf{Step 2}}
    By an argument similar to the 
  proof of Proposition \ref{U_min preserving} for \(Proj_1(H)\) instead of \(\cU_{\min}(H)\), we have 
    \[
    V_0(Proj_1(H))=Proj_1(H).
    \]
    Arguing similarly to the argument in Theorem \ref{main}, by replacing   Lemma \ref{lem:trace-recovery} with  Lemma \ref{tr-preserving for Proj}, we have
    \[
    {\rm Tr}(ef)={\rm Tr}(V_0(e)V_0(f)),
    \]
    for all \(e,f \in Proj_1(H)\). Hence, by \cite[Theorem 3.1]{Cassinelli}, there is either a unitary or an antiunitary operator \(U\) on \(H\) such that 
    \begin{align}\label{positive V_0}
    V_0(e)=UeU^*,\quad e\in Proj_1(H).
    \end{align}
    
    {\bf{Step 3}}
    We then prove that \eqref{positive V_0} holds for arbitrary \(x\in S(C_p(H)^+)\).
    Since Lemma \ref{rank-one decide+} holds for \(S(C_p(H)^+)\), it follows from the  same argument as in Theorem \ref{main}  that 
    \[V_0(x)=UxU^*,\quad x\in S(C_p(H)^+).\]
\end{proof}


   


\appendix
\section{Orthogonality of operators in noncommutative $L_p$-spaces}

Let $\cM$ be a von Neumann algebra on a  Hilbert space $\cH$, equipped with a semi-finite  faithful normal trace $\tau$.
A closed and densely defined operator $x$, affiliated with $\mathcal{M}$, is called \emph{$\tau$-measurable} if
$\tau(e ^{|x|}(s,\infty))<\infty$ for sufficiently large $s$, where $e ^{|x|}$ denotes  the spectral measure of $|x|$\cite{LSZ,DPS}.
For every $\tau$-measurable operator $x$ the {\it singular value function} $\mu(x)$ is defined by setting \cite{DPS,LSZ} 
$$\mu_t(x)=\inf\{s\geq0:\ \tau(e ^{|x|}(s,\infty))\leq t \},\quad t>0.$$
For \(0<p<1\), \(L_p(\cM,\tau)\) is defined as the set of all $\tau$-measurable operators \(x\) such that
\[
\|x\|_p
=
\left(
\int_0^\infty \mu_t(x) \,dt
\right)^{1/p}
<
\infty,
\]
(see \cite[Corollary 2.8]{Fack}).

Below, we  extend  \cite[Proposition A.2]{RX}
   to arbitrary operators in noncommutative \(L_p(\cM,\tau)\), $0<p<1$. 
The particular case of the following lemma for the Schatten $p$-classes, $p<1$, was claimed in \cite[Theorem 2.8]{McCarthy}. 

Our argument is motivated by  the   reduction scheme used by Bourin in the proof
of \cite[Theorem 2.1]{Bourin}: one first proves the assertion on the
positive cone, then extends it to self-adjoint operators via the decomposition
\(x=x_+-x_-\), and finally treats arbitrary operators through the standard
\(2\times2\) self-adjoint dilation.
\begin{lemma}\label{orth for p<1}
Let $\cM$ be a semifinite von Neumann algebra equipped with a faithful normal semifinite trace \(\tau\), and let $0<p<1$. For $x,y\in L_p(\mathcal M,\tau)$, 
 \[x\perp y\]
 if and only if 
 \[\left\|x+y\right\|_p^p=\|x\|_p^p+\|y\|_p^p.\]
\end{lemma}

\begin{proof}
For \(x,y \in L_p(\cM,\tau)\) with \(x\perp y\),
by \cite[Fact 1.3(i)]{RX}, we have that
\[
\left\|x+y\right\|_p^p=\|x\|_p^p+\|y\|_p^p.
\]
It remains to prove the sufficiency.

{\bf{Step 1}} We first prove the sufficiency for self-adjoint elements. Let \(x,y\in L_p(\cM,\tau)\) satisfy 
\[x=x^*,\quad y=y^*,\]
and
\begin{align}\label{x+y p}
    \left\|x+y\right\|_p^p=\|x\|_p^p+\|y\|_p^p.
\end{align}
Write
\begin{align}\label{x=x*}
 x=x_+-x_-,\quad y=y_+-y_-,\quad h=x+y.
\end{align}
Let
\[
 e=s(h_+),\qquad f=s(h_-)
\]
be the support projections of the positive and negative parts of \(h\).
Then, we have  \(e\perp f\), \(h_+=ehe\), and \(h_-=-fhf\). Hence, 
\begin{align}\label{h_+}
h_+= ex_+e+ey_+e -ex_-e-ey_-e
 \le ex_+e+ey_+e. 
\end{align}
Since the function \(t\mapsto t^p\) is concave on \([0,\infty)\), it follows from \cite[Proposition 4.6(i)]{Fack} that 
\begin{align}\label{h_+ p}
    \|h_+\|_p^p\stackrel{\eqref{h_+}}{\le} \|ex_+e+ey_+e\|_p^p\le
    \|x_++y_+\|_p^p
    \le \|x_+\|_p^p+\|y_+\|_p^p.
\end{align}
Similarly, we obtain that
\begin{align}\label{h_- p}
    \|h_-\|_p^p\le \|fx_-f+fy_-f\|_p^p\le \|x_-+y_-\|_p^p \le \|x_-\|_p^p+\|y_-\|_p^p.
\end{align}
Note that 
\[
\|x+y\|_p^p\stackrel{\eqref{x=x*}}{=}\|h\|_p^p=\|h_+\|_p^p+\|h_-\|_p^p.
\]
Combining \eqref{x+y p}, \eqref{h_+ p}, and \eqref{h_- p}, we have
\begin{align*}
\|x+y\|_p^p= \|h_+\|_p^p +\|h_-\|_p^p &\le \|x_++y_+\|_p^p+\|x_-+y_-\|_p^p\nonumber \\
&\le  \|x_+\|_p^p+\|y_+\|_p^p+\|x_-\|_p^p+\|y_-\|_p^p
\\
&=\norm{x}_p^p +\norm{y}_p^p =\|x+y\|_p^p \nonumber.
\end{align*}
Therefore, we have 
\begin{align*} 
    \|h_+\|_p^p= \|ex_+e+ey_+e\|_p^p=
    \|x_++y_+\|_p^p
  = \|x_+\|_p^p+\|y_+\|_p^p
\end{align*}
and 
\begin{align*} 
    \|h_-\|_p^p=\|fx_-f+fy_-f\|_p^p=  \|x_-+y_-\|_p^p= \|x_-\|_p^p+\|y_-\|_p^p.
\end{align*}
By \cite[Proposition A.2]{RX}, we have 
\begin{align}\label{x_+ perp y_+}
x_+\perp y_+,\quad x_-\perp y_-.
\end{align}
We claim that $e\ge s(x_++y_+ )$. 
Otherwise, we have $(x_++y_+)^{1/2}(I -e)(x_++y_+)^{1/2}>0$. It follows from \cite[Proposition 2.2]{Chilin} that
\begin{align*}
\mu(x_+ +y_+)&= \mu((x_++y_+)^{1/2}(I-e+e)(x_++y_+)^{1/2}) \\& >\mu((x_++y_+)^{1/2}e(x_++y_+)^{1/2})
=\mu(e(x_++y_+)e).   
\end{align*}
Therefore, $\norm{x_++y_+}_p> \norm{e(x_++y_+)e}_p$, a contradiction. 
The same argument yields that $f\ge s(x_-+y_-)$. 
In other words, 
\begin{align}\label{e(x_++y_+)e = x_++y_+}
e(x_++y_+)e = x_++y_+,\quad f(x_-+y_-)f = x_-+y_-. 
\end{align}
By  \eqref{e(x_++y_+)e = x_++y_+}, we have
\[
x_+=ex_+e,\quad y_+=ey_+e,\quad x_-=fx_-f,\quad y_-=fy_-f,
\]
which together with $ef=0$ implies that
\begin{align}\label{x_+ perp y_-}
x_+\perp y_-,\quad x_-\perp y_+.
\end{align}
Combining \eqref{x_+ perp y_+} and \eqref{x_+ perp y_-}, we have that
\[
x\perp y.
\]
This proves the self-adjoint case.

{\bf{Step 2}} We now prove the sufficiency for arbitrary \(x,y\in L_p(\cM,\tau)\) satisfying
\begin{align}\label{x+y p arbitrary}
    \left\|x+y\right\|_p^p=\|x\|_p^p+\|y\|_p^p.
\end{align}
In \(L_p(M_2(\cM,\tau))\), set
\[
 X=
 \begin{pmatrix}
 0&x\\
 x^*&0
 \end{pmatrix},
 \qquad
 Y=
 \begin{pmatrix}
 0&y\\
 y^*&0
 \end{pmatrix}.
\]
Then,  \(X=X^*\) and \(Y=Y^*\). Moreover,
\[
 \|X\|_p^p=2\|x\|_p^p,\qquad
 \|Y\|_p^p=2\|y\|_p^p,
\]
and then 
\[
 \|X+Y\|_p^p=2\|x+y\|_p^p\stackrel{\eqref{x+y p arbitrary}}{=}2\|x\|_p^p+2\|y\|_p^p=\|X\|_p^p+\|Y\|_p^p.
\]
By Step 1, we have that
\[
 X\perp Y,
\]
which implies that
\[
XY=0.
\]
Since
\[
 XY=
 \begin{pmatrix}
 xy^*&0\\
 0&x^*y
 \end{pmatrix},
\]
we have that
\[
 xy^*=0,\qquad x^*y=0,
\]
which implies that
\[
x\perp y.
\]
This completes the proof.  
\end{proof}

\bibliographystyle{amsalpha}

\end{document}